\documentclass[12pt]{amsart}
\usepackage[a4paper, margin=1in]{geometry}

\usepackage{amsmath,amsthm}
\usepackage{amsfonts,amstext,amssymb, mathtools, comment}

\usepackage{amsfonts}
\usepackage{amssymb}
\usepackage{amsmath}
\usepackage{enumerate}
\usepackage{rotating}
\usepackage{graphicx}
\usepackage{bm}
\usepackage{graphicx}
\usepackage{color}

\newcommand{\beqn}{\vspace{-0.25cm}\begin{eqnarray*}}

\newcommand{\eeqn}{\end{eqnarray*}}
\newcommand{\bneqn}{\vspace{-0.25cm}\begin{eqnarray}}
\newcommand{\eneqn}{\end{eqnarray}}

\newcommand{\p}{{\mathbb P}}

\newcommand{\e}{{\mathbb E}}
\newcommand{\D}{{\mathrm d}}
\newcommand{\ee}{{\mathrm e}}

\renewcommand{\a}{{\alpha}}

\newcommand{\ep}{\varepsilon}

\newcommand{\n}{{(n)}}

\newcommand{\X}{{\widehat X}}
\renewcommand{\aa}{{\widehat a}}

\newcommand{\ind}[1]{\mbox{\rm\large  1}_{\{#1\}}}

\newtheorem{theorem}{Theorem}
\newtheorem{proposition}{Proposition}

\newtheorem{lemma}{Lemma}
\newtheorem{remark}{Remark}
\newtheorem{Algorithm}{Algorithm}

\begin{document}

\author[S. Asmussen and J. Ivanovs]{S\o{}ren Asmussen \and
        Jevgenijs Ivanovs 
}






\title{Discretization error\\ for a two-sided reflected L\'evy process}




\renewcommand{\theequation}
{\arabic{section}.\arabic{equation}}

\begin{abstract}
	An obvious way to simulate a L\'evy process $X$ is to sample its increments over time $1/n$, thus constructing an approximating random walk $X^\n$. This paper considers the error of such approximation after the two-sided reflection map is applied, with focus on the value of the resulting process $Y$ and regulators $L,U$ at the lower and upper barriers at some fixed time. 
	Under the weak assumption that $X_\ep/a_\ep$ has a non-trivial weak limit for some scaling function  $a_\ep$ as $\ep\downarrow 0$, it is proved in particular that  $(Y_1-Y^\n_n)/a_{1/n}$  converges weakly to~$\pm V$, where the sign depends on the last barrier visited.
Here	the limit $V$ is the same as in the problem concerning approximation of the supremum as recently described by Ivanovs (2017). Some further insight in the distribution of $V$ is  provided both theoretically and numerically.  
\end{abstract}

\maketitle

\section{Introduction}
Consider a non-monotone L\'evy process $X$ and its two-sided reflection with respect to the interval $[0,b]$ and initial position $x$:
\begin{equation}\label{eq:reflection}Y_t=x+X_t+L_t-U_t,\qquad t\geq 0.\end{equation}
Here $(Y,L,U)$ is the solution of the Skorokhod problem, i.e.\ $L$ and $U$ are the unique non-decreasing processes (often called regulators at Lower and Upper barriers) such that 
$Y_t\in[0,b]$ and $\int_{[0,\infty)} \ind{Y_t>0}\D L_t=\int_{[0,\infty)} \ind{Y_t<1}\D U_t=0$. Importantly, the two-sided reflection can be constructed by alternating one-sided reflection at the upper and lower barriers~\cite[Ch.\ XIV.3]{APQ}, and so $-L$ and $U$ evolve locally as the running infimum and supremum of $X$, respectively. In queueing theory the process $Y$ is commonly used to model a queue with a finite buffer, whereas in collective risk theory $U,L$ and $Y$ are often interpreted as cumulative dividends, capital injections and the resulting surplus process, respectively. 

A survey and a comprehensive list of references is given in \cite{LM} with emphasis on steady-state
features such as the stationary distribution and the loss rate $\lim_{t\to\infty}U_t/t$. For a textbook reference on L\'evy-driven queues, see~\cite{debicki_mandjes}.
The focus of this paper is on the distribution of $Y_T$ for a fixed time horizon $T>0$. Without loss of generality we may  (and do) fix $b=1$ and $T=1$, because under appropriate rescaling of time and space the new process $X'_t=X_{Tt}/b$ is again a L\'evy process. 

 Let $X^\n_i=X_{i/n}$ for $i=0,\ldots, n$ be the random walk approximating~$X_t, t\in[0,1]$, which similarly leads to the two-sided reflected (discrete-time) process:
\begin{equation}\label{eq:reflection_discr}Y^\n_i=x+X^\n_i+L^\n_i-U^\n_i,\end{equation}
alternatively described via the recursion
\begin{equation*}Y^\n_i=\max\bigl(0,\min(1,Y^\n_{i-1}+X^\n_i-X^\n_{i-1})\bigr)
\end{equation*}
for $i\geq 1$.
Our main quantity of interest is $\Delta^\n=Y_1-Y^\n_n$, the discretization error corresponding to the distribution of $Y_1$, see Figure~\ref{fig:ex100}. 

\begin{figure}[htb]
	\centering
	\includegraphics[width=\textwidth]{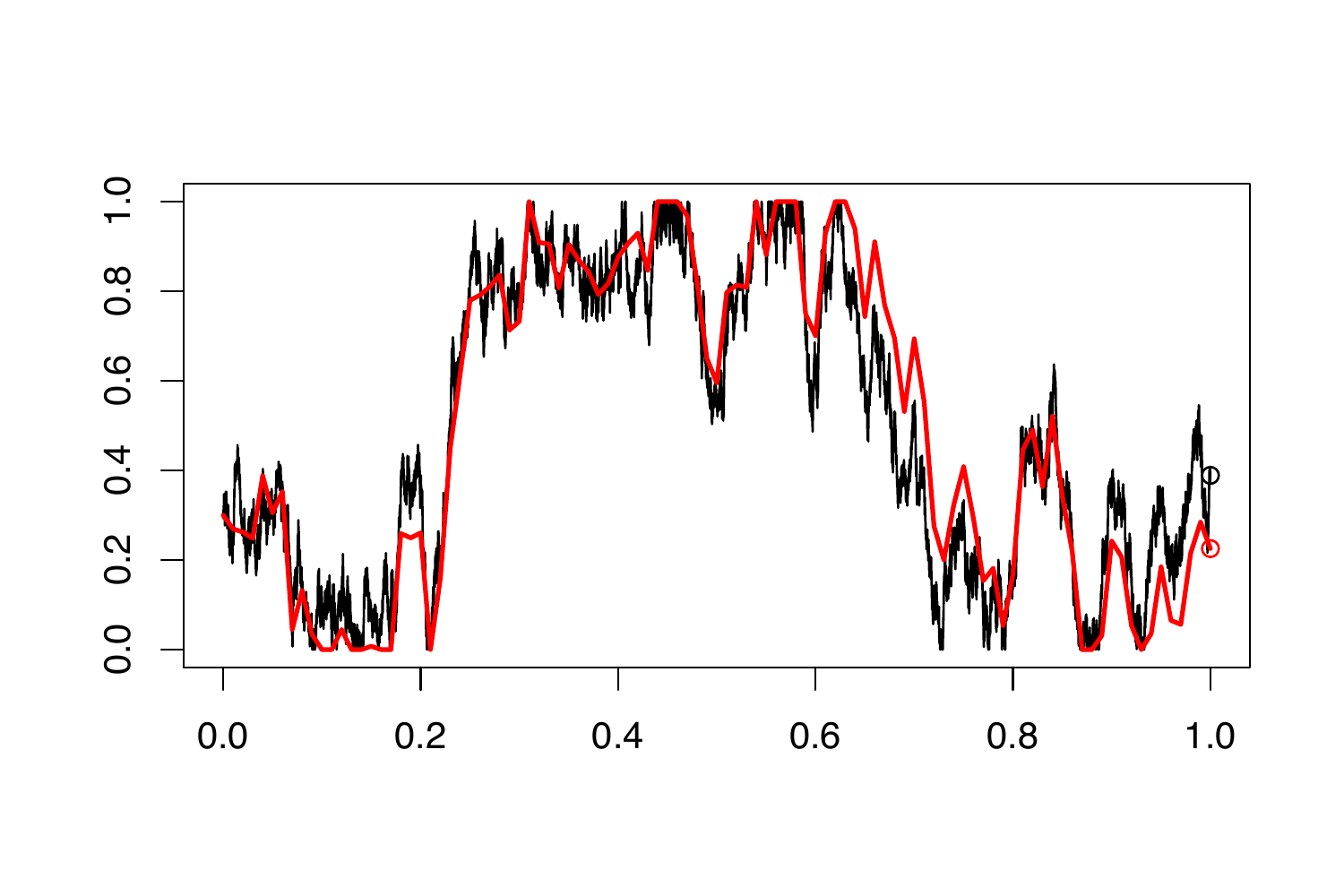}
	\vspace{-1.5cm}
	\caption{Example of a reflected sample path (black) and its discretized version (red)  for $n=100$.}
	\label{fig:ex100}
\end{figure}

A key observation for understanding the impact
of discretization is the formulas
\begin{equation}\label{0107a} X_{i/n}-\inf_{s\le i/n} X_s,\quad X^\n_i-\min_{j\le i } X^\n_j
=X_{i/n}-\min_{j\le i}X_{j/n}\end{equation}
for the one-sided reflected version started at $0$ and its discretized version. They show that  
the discretized version is pushed less away from $0$ and hence will stay below the
continuous time version. For the two-sided problem, explicit formulas expressing $Y,L,U$ in terms of $X$ also exist, see \cite{LM} for a summary,   but they seem too complicated
to be of much practical use. Nevertheless, in view of \eqref{0107a} one may expect for the two-sided reflected process that the sign of the error $\Delta^\n$ for large enough $n$ depends on the last barrier visited.
This is indeed confirmed by Figure~\ref{fig:ex100}. Note, however, that  switching between the barriers may occur with a rather long delay, and for smaller values of $n$ the pattern of switching is likely to be different for the two processes; we return to this point later.

Further intuition is provided by the easy observation that the limiting behaviour of the discretization error must also depend on the local behaviour of the underlying L\'evy process~$X$: the error rate is bigger when $X$ is more volatile. Thus one may expect that processes with a non-zero Brownian component are the worst ones to discretize in terms of the error rate, which is indeed true as explained in Section~\ref{sec:main}. This paper heavily relies on the recent results in~\cite{ivanovs_zooming} concerning discretization error when computing the supremum of a L\'evy process, which in its turn employs the machinery of weak convergence of stochastic processes~\cite{whitt}.
Here we avoid using technical concepts from~\cite{ivanovs_zooming} and mainly focus on the consequences for the intricate two-sided reflection.

Notationally, the event that the lower barrier was visited last will be written as $\{\rho^\n_L>\rho^\n_U\}$, where
\[\rho^\n_L=\min\{i\leq n:L^\n_i=L^\n_n\},\qquad \rho^\n_U=\min\{i\leq n:U^\n_i=U^\n_n\},\] are the last times when $L^\n$ and $U^\n$ increase, respectively. Note that $\rho^\n_L=\rho^\n_U=0$ corresponds to the scenario when reflection has not been applied up to and including step~$n$. 
Finally, let $N^\n$ be the number of switches between the barriers counting the first time when one of the regulators increases. More precisely, it is the largest number $k$ such that there exists a sequence of times $0<i_1<\cdots<i_k\leq n$ at which $L^\n$ or $U^\n$ increase and do so alternatingly. 
We use $\rho_L,\rho_U$ and $N$ to denote the corresponding quantities of the continuous time system. 
In some cases it is possible that the original reflected process touches the barrier and then leaves it without necessitating increase of the relevant regulator; such epochs are not considered as barrier switching epochs.

\section{The main result}\label{sec:main}
The basic assumption underlying our main result is that the following weak convergence holds: 
\begin{equation}\label{eq:zooming}X_{\ep}/a_\ep\stackrel{d}{\to} \X_1\qquad\text{ as }\ep\downarrow 0\end{equation}
for some scaling function $a_\ep>0$ and a random variable $\X_1$, not identically 0, which then must be infinitely divisible~\cite[Thm.\ 15.12(ii)]{kallenberg}. This convergence readily extends to convergence of processes $(X_{t\ep}/a_\ep)_{t\geq 0}\stackrel{d}{\to}  (\X_t)_{t\geq 0}$ on the Skorokhod space~\cite[Thm.\ 16.11]{kallenberg}, which can be seen as zooming-in on the process~$X$.
Necessarily, $\X$ is $1/\a$-self-similar L\'evy process with $\a\in(0,2]$ and then $a_\ep$ is regularly varying at 0 with index ${1/\a}$. Thus there are only the following possibilities for the limit process $\X$: (i) (driftless) Brownian motion with $\a=2$, (ii) linear drift  with $\a=1$ or (iii) strictly $\a$-stable L\'evy process with $\a\in(0,2)$.  A complete characterization of the respective domains of attraction can be found in~\cite{ivanovs_zooming}. 

\begin{remark}
It is noted that~\eqref{eq:zooming} is a weak regularity assumption satisfied for almost every L\'evy processes of practical interest.
In particular, it is always satisfied if the Brownian component is present, i.e.\ $\sigma>0$, irrespective of the L\'evy measure. In this case, one can take $a_\ep=\sigma\sqrt{\ep}$ so that the limit process $\X$ is a standard Brownian motion. Moreover,~\eqref{eq:zooming} is satisfied by processes of bounded variation on compacts with non-zero drift component, in which case $\X$ is a linear drift process. In all the other cases, one needs to look at the behaviour of the L\'evy measure $\Pi(\D x)$ at~0, see~\cite{ivanovs_zooming} for details, and the limit may be any of (i) -- (iii) mentioned above. A sufficient condition is that the functions $\Pi(x,\infty),\Pi(-\infty,-x)$ are regularly varying at 0 with index $-\a$ for $\a\in (0,1)\cup(1,2]$ and their ratio has a limit in $[0,\infty]$.
Finally, assumption~\eqref{eq:zooming} is not satisfied by the following standard classes of L\'evy processes: (a) driftless compound Poisson process (trivial case) and its neighbour (b) driftless variance gamma process; in both cases the functions $\Pi(x,\infty),\Pi(-\infty,-x)$ are slowly varying at~0.
\end{remark}

In~\cite{ivanovs_zooming} it is shown that under the assumption~\eqref{eq:zooming} it holds that
\begin{equation}\label{eq:supremum}\frac{1}{a_{n^{-1}}}\left(\sup_{t\in[0,1]}X_t-\max_{i=0,\ldots, n} X_{i/n}\right)\stackrel{d}{\to}  V>0,\qquad n\to\infty\end{equation}
on the event that the supremum is not achieved at the endpoints of the interval $[0,1]$; the case of a Brownian motion $X$ with drift was analyzed long before in~\cite{agp}. 
The limiting random variable $V$ is defined using the laws of $\X$ `conditioned' to be positive and `conditioned' to be negative and an independent uniform time shift,  see~\cite{ivanovs_zooming} for the exact definition. In particular, when $\X$ is  standard Brownian motion we may take $V=\min_{i\in\mathbb Z}B_{\Upsilon+i}$ as in~\cite{agp}, where $(B_t)_{t\geq 0},(B_{-t})_{t\geq 0}$ are two independent copies of a 3-dimensional Bessel process and $\Upsilon$ is an independent $[0,1]$-uniform random variable.  Importantly, the analogous result concerning discretization error of the infimum holds true, that is, in~\eqref{eq:supremum} we may replace $\sup,\max,V$ by $\inf,\min,-V$, respectively. Some further comments and numerical illustrations concerning the distribution of $V$ are provided in Section~\ref{sec:V}.

With a small abuse of terminology we say that $X$ is regular for $(0,\infty)$ if it enters $(0,\infty)$ immediately, and irregular otherwise. Precise conditions for regularity can be found in, for example,~\cite[Thm.\ 6.5]{kyprianou}. In particular, any $X$ of unbounded variation on compacts is regular for both half lines $(0,\infty)$ and $(-\infty,0)$. Moreover, if $X$ is irregular for one half line then it must be regular for the other unless $X$ is a driftless compound Poisson process, which is excluded from the following by assuming~\eqref{eq:zooming}. The importance of regularity is underlined by the well known fact that $Y_1$ has mass at 0 or at 1 if $X$ is irregular for $(0,\infty)$ or for $(-\infty,0)$, respectively. 
 In addition, it is shown in~\cite{ivanovs_zooming} that 
if $X$ is irregular for $(0,\infty)$ (for $(-\infty,0)$) and~\eqref{eq:zooming} holds, then necessarily $\X$ is decreasing (increasing), and so the limit distribution in~\eqref{eq:supremum} has a simple form: $V\stackrel{d}{=}|\X_\Upsilon|$.
We are now ready to state our main result.

\begin{theorem}\label{thm:main}
Assume that $X$ is regular for both half lines and that~\eqref{eq:zooming} holds. Then as $n\to\infty$ we have that
\begin{align}\label{eq:res1}&\left.\frac{1}{a_{n^{-1}}}\left(Y_1-Y^\n_n\right)\right|\left\{\rho_L^\n>\rho_U^\n\right\}\stackrel{d}{\to}  V,\\
\label{eq:res1b}&\left.\frac{1}{a_{n^{-1}}}\left(Y_1-Y^\n_n\right)\right|\left\{\rho_U^\n>\rho_L^\n\right\}\stackrel{d}{\to}  -V.
\end{align}
Furthermore, for any $k\geq 1$ and $Y_0\in (0,1)$
\begin{align}\label{eq:res2}
\left.\frac{1}{a_{n^{-1}}}\left(L_1-L^\n_n,U_1-U^\n_n\right)\right|&\left\{N^\n=k,\rho_L^\n>\rho_U^\n\right\}\\ &\qquad\stackrel{d}{\to} 
 \left(\sum_{i=1}^kV_i,\sum_{i=1}^{k-1}V_i\right),\nonumber\\
 \label{eq:res2b}\left.\frac{1}{a_{n^{-1}}}\left(L_1-L^\n_n,U_1-U^\n_n\right)\right|&\left\{N^\n=k,\rho_U^\n>\rho_L^\n\right\}\\ &\qquad\stackrel{d}{\to} 
 \left(\sum_{i=1}^{k-1}V_i,\sum_{i=1}^kV_i\right),\nonumber
\end{align}
where the $V_i$ are independent and distributed as $V$. Moreover, the convergence in the above statements is mixing in the sense of R\'enyi~\cite{renyi}, i.e.\ these results hold when additionally conditioning on any event $B$ of positive probability.

If $X$ is irregular for $(0,\infty)$ then the above results hold true if: 
in~\eqref{eq:res1} and in~\eqref{eq:res2} we additionally condition on $Y^\n_n\neq 0$, and in~\eqref{eq:res2} and in~\eqref{eq:res2b} 
we additionally condition on 
\[S^\n=\left\{\not\exists i\in [1,n-1]:Y^\n_{i-1}=0,Y^\n_{i}=1,Y^\n_j<1\, \forall j\in [i+1,n_i]\right\},\]
where $n_i=\min\{j\in[i+1,n]:Y_j^\n=0\text{ or }j=n\}$.

The analogous statement is true in the case when $X$ is irregular for $(-\infty,0)$.
\end{theorem}

Let us provide some comments concerning Theorem~\ref{thm:main}. 
Firstly, for the regulators $L$ and $U$ the discretization error accumulates as the number of switches between the barriers grows, see~\eqref{eq:res2} and~\eqref{eq:res2b} . These additional errors, however, cancel out when computing~$\Delta^\n$, see~\eqref{eq:reflection} and~\eqref{eq:reflection_discr}. Intuitively, this error regenerates at every barrier switching epoch. 

Importantly, the events upon which we condition can be modified into the corresponding events for the original process without affecting the result. That is, in~\eqref{eq:res1} we may instead condition on $\{\rho_L>\rho_U\}$,  and on $\{\rho_L>\rho_U,Y_1\neq 0\}$ when $X$ is irregular for $(0,\infty)$. 
The event $S$ corresponding to $S^\n$ asserts that there does not exists a time $t\in(0,1)$ such that $t$ is a point of increase of both $L$ and $U$ and that following this time $U$ does not increase before $L$ does.
In other words, we need to exclude the possibility that $Y$ jumps from the lower barrier to the upper barrier (necessitating increase of $U$) and then hits the lower barrier again (or the time runs out), because in this case there will be a certain dependence between $V_i$'s, see the proof for details.
When $X$ is regular for $(0,\infty)$ then $\{Y_1\neq 0\}$ and $S$ hold with probability 1 and so the additional conditioning is not required. 
Also if $X$ is irregular for $(0,\infty)$ and its positive jumps are bounded by~1 then $S$ holds, and again there is no need for extra conditioning on~$S^\n$. 

Observe that~\eqref{eq:res1} and~\eqref{eq:res1b} cover the cases when a barrier was hit and the process at its terminal value is away from the barrier. If this is not the case then the discretization scheme employed does not produce any error with respect to $Y_1$ for large enough~$n$. 
Furthermore, adding more information does not affect the limiting error distribution, since the convergence is mixing.
 Finally, the event $\{N=k,\rho_L>\rho_U\}$ has 0 probability for $k$ even (odd) when $Y_0=0$ ($Y_0=1$), and this is the only reason to exclude $Y_0\in\{0,1\}$ in the statements~\eqref{eq:res2} and~\eqref{eq:res2b}.

\begin{remark}
	Theorem~\ref{thm:main} readily extends to the setting of the so-called two-sided refraction of $X$, see~\cite{power_identities}. Here refraction from above at rate $\gamma_U\in[0,1]$ in risk theory has the interpretation of taxation at rate $\gamma_U$ according to a loss-carry-forward scheme.
\end{remark}

Importantly, Theorem~\ref{thm:main} provides a way to improve the simulation results obtained through discretization. 
In particular, relying on the fact that the convergence in~\eqref{eq:res1} and~\eqref{eq:res1b} is mixing, we propose the following procedure to rectify the samples of $Y^\n_n$ in the case when $X$ is regular for both half lines.
\begin{Algorithm}\label{alg}
	\hfill
\begin{itemize}
	\item simulate $(Y_n^{(n)},\rho_L^{(n)},\rho_U^{(n)})$,
	\item add independent realizations of $a_{n^{-1}} V$ to those of $Y_n^{(n)}$ with $\rho_L^{(n)}>\rho_U^{(n)}$
	\item add independent realizations of $-a_{n^{-1}} V$ to those of $Y_n^{(n)}$ with $\rho_L^{(n)}<\rho_U^{(n)}$
	\item return the rectified realizations of $Y_n^{(n)}$.
\end{itemize}
\end{Algorithm}
Regularity for both half lines implies that conditioning on $Y^\n_n$ being away from the boundaries is not needed.
This extra conditioning is irrelevant for the limit theory, but makes a difference when using moderate~$n$. Furthermore, the simulated realizations produced by Algorithm~\ref{alg} may lie outside of $[0,1]$, and so further post-processing should be considered.
In Section~\ref{sec:numerics} we provide some first numerical results demonstrating effectiveness of the proposed algorithm.

\section{Proof of  the main result}
Our proof of Theorem~\ref{thm:main} relies on the following generalization of~\eqref{eq:supremum}. 
\begin{lemma}
\label{lem:zooming}
Assume~\eqref{eq:zooming} and consider two random times $0\leq \rho<\tau<\infty$.
Then it holds that
\begin{equation}\label{eq:sup2}\left.\frac{1}{a_{n^{-1}}}\left(\sup_{t\in[\rho,\tau]} X_t-\max_{i=\lceil\rho n\rceil,\ldots, \lfloor \tau n\rfloor} X_{i/n}\right)\right|B\stackrel{d}{\to}  V, \qquad\text{as }n\to\infty\end{equation}
for any positive probability event $B\in\mathcal F$ such that on $B$ the supremum is achieved strictly inside the interval $[\rho,\tau]$, i.e.
$B\subset\{\sup_{t\in[\rho,\tau]} X_t>X_{\rho}\vee X_{\tau-}\vee X_{\tau}\}$.

The same limit is obtained if $\sup-\max$ is replaced by $-\inf+\min$ assuming that on $B$ the infimum is achieved strictly inside the interval $[\rho,\tau]$.

Moreover, for $k$ random intervals $[\rho_i,\tau_i]$ such that $0\leq\rho_1<\tau_1\leq \cdots\leq \rho_k<\tau_k<\infty$ there is the joint convergence of the corresponding rescaled errors to the vector $(V_1,\ldots,V_k)$ with independent components, where on $B$ the suprema/infima are achieved strictly inside the respective intervals. 
\end{lemma}
\begin{proof}
The proof of~\eqref{eq:supremum} in~\cite{ivanovs_zooming} shows, in fact, that~\eqref{eq:sup2} holds for deterministic $\rho,\tau$ and any event $B$ of positive probability. 
If all the times $\rho_i,\tau_i$ are deterministic then the conclusion of the lemma follows from the fact that $X$ has independent increments.

Pick $\delta>0$ and consider the event that for all $i$ the supremum over $[\rho_i,\tau_i]$ is achieved in $[\rho_i+\delta,\tau_i-\delta]$; the intersection of this latter event and $B$ will be denoted by~$B_\delta$.
Note that $\p(B\backslash B_\delta)\to 0$ as $\delta\downarrow 0$ and hence it is enough to show the claimed convergence on the event $B_\delta$ for all $\delta>0$ assuming it has positive probability. Now consider (random) integers $s_i=\lceil\rho_i/\delta\rceil,t_i=\lfloor \tau_i/\delta\rceil$ and note that the supremum over $[\rho_i,\tau_i]$ must be achieved strictly inside $[s_i\delta,t_i\delta]\subset[\rho_i,\tau_i]$ on $B_\delta$. Note also that the limit result does not change if we restrict ourselves to the intervals $[s_i\delta,t_i\delta]$. Finally, it is left to condition on the values of all $s_i,t_i$ and to apply the result for deterministic times.
\qed
\end{proof}


\begin{proof}[of Theorem~\ref{thm:main}] 
In view of~\eqref{eq:reflection} and~\eqref{eq:reflection_discr} we obtain
\[Y_1-Y_n^\n=(L_1-L^\n_n)-(U_1-U_n^\n),\]
because our discretization scheme implies that $X^\n_n=X_1$.
Thus it is sufficient to establish~\eqref{eq:res2}, since~\eqref{eq:res2b} follows similarly, and these two results readily imply~\eqref{eq:res1} and~\eqref{eq:res1b}, respectively.

It is well known~\cite[Eq.\ (2.2.13)]{jacod_protter} that the discretized process $X^\n_{\lfloor tn\rfloor}=X_{\lfloor tn\rfloor/n}$ converges to $X_t$  in the Skorokhod topology for every sample path. Moreover, the processes $Y^\n_{\lfloor tn\rfloor},L^\n_{\lfloor tn\rfloor},U^\n_{\lfloor tn\rfloor}$ converge to $Y_t,L_t,U_t$, respectively, on the level of sample paths, which can be deduced from the fact that the running supremum and one-sided reflection maps are continuous~\cite[Thm.\ 13.4.1/13.5.1]{whitt}. 
Furthermore, letting 
\[B^\n_k=\{N^\n=k,\rho_L^\n>\rho_U^\n\}, \qquad B_k=\{N=k,\rho_L>\rho_U\}\] we observe that $\p(B^\n_k\bigtriangleup B_k)\to 0$ as $n\to\infty$,
and so it is sufficient to prove~\eqref{eq:res2} when conditioning on the event $B_k$ instead of $B_k^\n$.
In the following, without real loss of generality, we also assume that the lower barrier is hit first, that is, $L$ becomes positive before $U$, and so we only consider events $B_k$ for $k$ odd.

Let $\tau_i$ for $i\leq k$ be the barrier switching epochs for the original process where $\tau_0=0$. More precisely, we define
\[\tau_1=\inf\{t>0:x+X_t+L_t> 1\}=\inf\{t>0:U_t>0\}\]
to be the first time when the upper barrier is hit by the process reflected at the lower barrier. 
Similarly define $\tau_2=\inf\{t>0:L_t>L_{\tau_1}\}$, $\tau_3=\inf\{t>0:U_t>U_{\tau_2}\}$ and so on.
Note that on $B_k$ we must have $\{\tau_{k-1}<1<\tau_k\}$ a.s.
Additionally to the above, we may assume that for $i$ odd there is an increase of $L^\n_j$ but none of $U^\n_j$ for $j=\lceil \tau_{i-1}n\rceil,\ldots \lfloor \tau_{i}n\rfloor$ and vice versa for $i$ even; this must happen with probability tending to 1 as $n\to\infty$.

By splitting the sample path into the intervals $[\tau_{i-1},\tau_i\wedge 1),i=1,\ldots,k$, we find under the above assumptions that 
\[L_1-L^\n_n=\sum_{i=1}^k V^\n_i,\qquad U_1-U^\n_n=\sum_{i=1}^{k-1} V^\n_i,\] where 
\begin{align*}
V^\n_i=\begin{cases}-\inf_{t\in[\tau_{i-1},\tau_i\wedge 1]} X_t+\min_{j=\lceil\tau_{i-1}n\rceil,\ldots,\lfloor(\tau_{i}\wedge 1)n\rfloor}X^\n_j,&i\text{ is odd},\\
\sup_{t\in[\tau_{i-1},\tau_i\wedge 1]} X_t-\max_{j=\lceil\tau_{i-1}n\rceil,\ldots,\lfloor(\tau_{i}\wedge 1)n\rfloor}X^\n_j,&i\text{ is even}.
\end{cases}
\end{align*} 
The main observation here is that the error of $L$ (of $U$) accumulated in each interval for $i$ odd (even) is composed of two parts: (a) the error of discretization assuming the discretized process starts from the respective barrier and (b) the difference in effective height of the barriers which comes from part (a) of the previous interval (it is 0 when $i=1$).
In particular, in the 1st interval the error for $L$ is indeed given by~$V_1^\n$, which implies that for the 2nd interval the upper barrier is effectively $V_1^\n$ units 
lower for continuously observed process. Hence the error for $U$ accumulated in the 2nd interval is given by $V_1^\n+V_2^\n$. For the 3rd interval the lower barrier is effectively $V_2^\n$ units higher for continuously observed process, and so the error for $L$ accumulated in the 3rd interval is $V_2^\n+V_3^\n$, which should be added to $V_1^\n$ accumulated in the 1st interval, and so on. Finally, note that $V_k^\n$ does not contribute to the error of $U$.

Regularity of $X$ for both half lines implies that infima and suprema in the intervals $[\tau_{i-1},\tau_i\wedge 1]$ are not achieved at the end points.
The proof is now complete in this case upon invoking Lemma~\ref{lem:zooming}. 

The case when $X$ is irregular for $(0,\infty)$ requires additional care. In this case it is possible that the infimum and supremum are achieved at the right and left end points of the intervals, respectively, but excluding times 0 and 1 (we assume that $Y_1\neq 0$ when $\rho_L>\rho_U$).
This does not create any problem with respect to $\Delta^\n$ which is the limit of $V_k^\n$, because we can extend the interval $[\tau_{k-1},1]$ to the left.
Similar reasoning goes through with respect to the errors for $L$ and $U$, unless these errors have a contribution from the same time $t=\tau_i$, i.e.~such $t$ is a point of increase of both $L$ and $U$ (in particular there is a jump up of size larger than 1) and that $U$ does not increase on $(\tau_i,\tau_{i+1}\wedge 1]$. On the latter event the limit of $(V_i^\n,V_{i+1}^\n)$ can be shown to be $(-\X_\Upsilon,-\X'_{1-\Upsilon})$, where $\X,\X',\Upsilon$ are independent and the former two have the same law (of a decreasing process). This means that the marginals have the law of $V$ but fail to be independent.
Finally, the complement of the above event is denoted by $S$ and one can show that $\p(S^\n\bigtriangleup S)\to 0$ as $n\to\infty$, which concludes the proof.  
\qed
\end{proof}



\section{Numerical illustration}\label{sec:numerics}
In this section we provide a numerical illustration of Theorem~\ref{thm:main} and the related Algorithm~\ref{alg} using a Brownian motion $X$ with  variance  and drift parameters $\sigma^2=2$ and $\mu=-1/2$, respectively. As pointed out in Section~\ref{sec:main} we may choose $a_\ep=\sigma\sqrt{\ep}$ so that $\X$ is standard Brownian motion, and then $V$ is obtained from two independent copies of a 3-dimensional Bessel process. In this work we use a  straightforward procedure to simulate $V$ based on sampling both copies at 300 locations, whereas an algorithm for exact simulation is given in~\cite{agp} which also provides a formula for~$\e V$.

We fix the initial position $x=0.3$ and choose $n=100$, so that the increments of the approximating random walk are distributed according to $\mathcal N(\mu/n,\sigma^2/n)$, see Figure~\ref{fig:ex100}. In fact, we first use $n'=50000$ increments and then accumulate them into $n=100$ increments hoping that the former discretization provides a very good approximation of $Y_1$. This approximation is slightly improved by adding and subtracting the constant $ a_{1/n'} \e V\approx 0.0037$ when $\rho^{(n')}_L>\rho^{(n')}_U$ and $\rho^{(n')}_U>\rho^{(n')}_L$, respectively, see again~\eqref{eq:res1} and~\eqref{eq:res1b}. Next, we simulate 20000 (approximate) realizations of the discretization error $\Delta^{(n)}$ and pick those with ${\rho_L^{(n)}>\rho_U^{(n)}}$ (about 60\%). The respective conditional distribution is compared to the distribution of $a_{n^{-1} }V$ suggested by the limit result in~\eqref{eq:res1}, see Figure~\ref{fig:err100}. As can be seen from Figure~\ref{fig:ex100}, the chosen discretization is rather coarse. Nevertheless, the corresponding discretization error is well captured by the limit result, apart from some mass on $(-\infty,0)$. This latter error comes from the failure to detect the last barrier being active, which is also demonstrated by Figure~\ref{fig:err100} where 'error adjusted' is obtained by further removing the realizations with $\rho_U>\rho_L$. In practice, the latter information is not available, and so one will need to increase~$n$. 
	\begin{figure}[htb]
		\centering
		\includegraphics[width=0.9\textwidth]{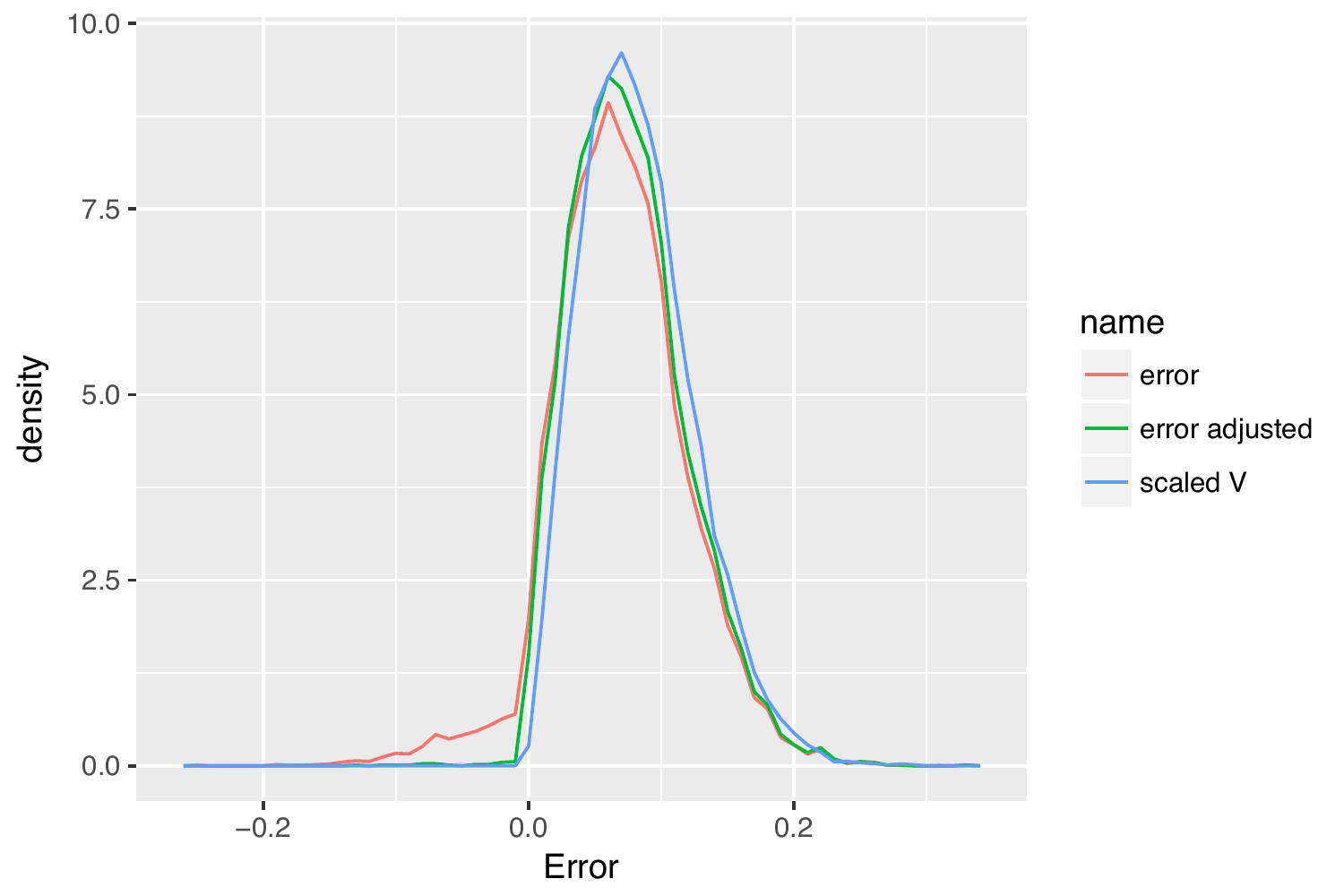}
		\caption{Distribution of $\Delta^\n\left|\{\rho_L^{(n)}>\rho_U^{(n)}\}\right.$ and its approximation $a_{n^{-1} }V$ for $n=100$.}
		\label{fig:err100}
	\end{figure}
Finally, Figure~\ref{fig:rec100} demonstrates the effectiveness of Algorithm~\ref{alg} based on the limit result in Theorem~\ref{thm:main}. 
\begin{figure}[htb]
	\centering
	\includegraphics[width=0.9\textwidth]{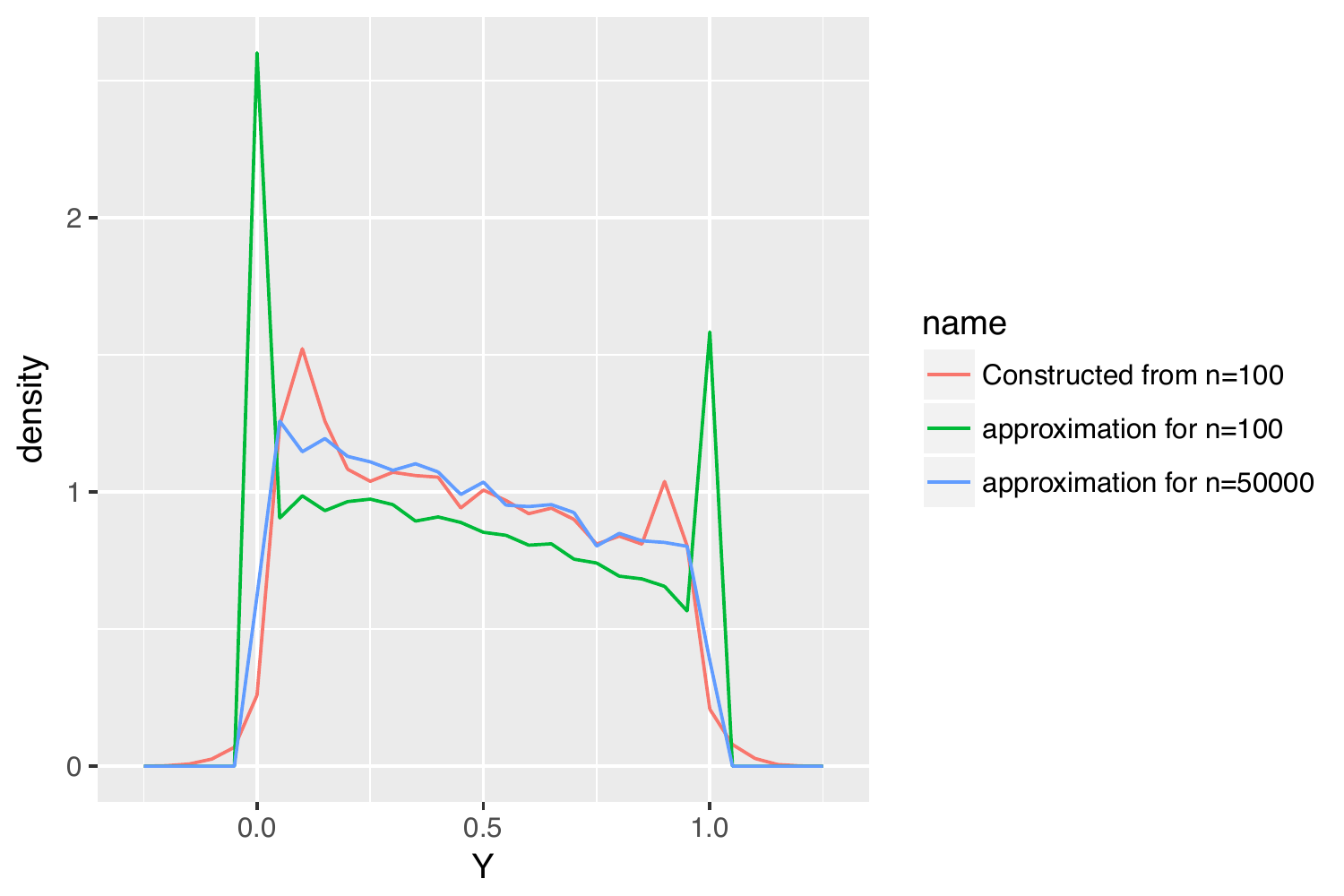}
	\caption{Approximation of $Y_1$. The red line corresponds to Algorithm~\ref{alg} with $n=100$.}
	\label{fig:rec100}
\end{figure}

\section{On the distribution of $V$}\label{sec:V}
Recall from~\cite{ivanovs_zooming} that the  
distribution of $V$ is defined in terms of the laws of $\X$ conditioned to be positive and conditioned to be negative and an independent uniform time shift, where the self-similar process $\X$ is one of the following: (i) standard Brownian motion, (ii) linear drift, (iii) strictly $\a$-stable process.
In the following we exclude monotone self-similar processes $\X$, i.e.\ linear drift processes and one-sided strictly $\a$-stable processes for $\a\in(0,1)$, since in this case we trivially have $V\stackrel{d}{=}|\X_\Upsilon|$ for an independent $[0,1]$-uniform~$\Upsilon$.
Moreover, the case (i) is addressed in~\cite{agp} and so we mainly concentrate on the case (iii) formed by a two-parameter family of strictly stable processes; the scale parameter is fixed, see also Proposition~\ref{Prop:0107a}. 

Recall that $\X$ arises as a limit in~\eqref{eq:zooming}. Importantly, $\X$ must be in its own domain of attraction: $\X_\ep/\aa_\ep\stackrel{d}{\to}  \X_1$, where according to~\cite[Thm.\ 2]{ivanovs_zooming} we may take $\aa_\ep=\ep^{1/\a}.$ Thus according to~\eqref{eq:supremum} the corresponding $V$ may also be obtained from
\begin{equation}\label{eq:V}
V_n\ =\ n^{1/\a}\left(\sup_{t\in[0,1]}\X_t-\max_{i=0,\ldots, n} \X_{i/n}\right)\stackrel{d}{=}\sup_{t\in[0,n]}\X_t-\max_{i=0,\ldots, n} \X_{i}\stackrel{d}{\to}  V
\end{equation}
as $n\to\infty$, where the equality in distribution follows from self-similarity of~$\X$.
Importantly, non-monotone self-similar $\X$ is regular for both half lines $(0,\infty)$ and $(-\infty,0)$ and so the supremum can not be achieved at the endpoints of the interval~$[0,1]$.
Furthermore, writing explicitly the dependence of $V$ on the law of the underlying process $\X$ we have:
\begin{proposition}\label{Prop:0107a}
$V_{\mathcal L(-\X)}\stackrel{d}{=}V_{\mathcal L(\X)}$ and for any $c>0$: $V_{\mathcal L(c\X)}\stackrel{d}{=}cV_{\mathcal L(\X)}$.
\end{proposition}
\begin{proof}
The last statement is obvious and the first readily follows from the representation of $V$
(see, however, also the last comment following~\eqref{eq:supremum}).
\qed
\end{proof}

Following the approach of~\cite{agp}, and additionally relying on self-similarity of~$\X$, we may identify the limit of $\e V_n$.  
It is assumed here that $\a\in(1,2]$ since otherwise $\e V_n=\infty$.
\begin{proposition}\label{prop:spitzer}
Consider the $1/\alpha$-self-similar L\'evy process $\X$ for $\alpha\in(1,2]$. It holds that
\[\e V_n\to-\zeta\left(\frac{\alpha-1}{\alpha}\right)\e \X^+_1,\]
where $\zeta$ is the Riemann zeta function and $\e \X^+_1$ for $\a\in(1,2)$ is given in~\eqref{eq:zolotarev} below.
\end{proposition}
\begin{proof}
Using Spitzer's identity   and self-similarity one readily finds that
\[\e \sup_{t\in[0,1]}\X_t=\int_0^1\frac{1}{t}\e \X^+_t\D t=\int_0^1\frac{1}{t} t^{1/\alpha}\D t\,\e \X^+_1=\alpha\e \X^+_1.\]
A similar calculation reveals that 
\[\e \max_{k=0,\ldots,n}\X_{k/n}=\sum_{k=1,\ldots,n}\frac{1}{k}\e\X^+_{k/n}=n^{-1/\alpha}\sum_{k=1,\ldots,n}k^{1/\alpha-1}\e\X^+_{1}.\]
Finally, we find using a well-known formula~\cite[(23.2.9)]{handbook} that
\[\e V_n=\left(\a n^{1/\a}-\sum_{k=1,\ldots,n}k^{1/\alpha-1}\right)\e\X^+_{1}\to-\zeta\left(\frac{\a-1}{\a}\right)\e\X^+_{1}.\]
\qed
\end{proof}

In the following we assume that $\X$ is a strictly $\a$-stable process with $\a\in (1,2)$, in which case the characteristic exponent can be written as
\begin{equation*}
\log \e\ee^{i\theta \X_1}=-|\theta|^\a\left(1-i\beta\tan\frac{\pi\a}{2}{\rm sgn}\theta\right),
\end{equation*}
where $\beta\in[-1,1]$ is the skewness parameter; the scale parameter is fixed to~1. The constants $c_\pm$ appearing in the corresponding L\'evy measure, see e.g.~\cite[Eq.~(5)]{ivanovs_zooming}, are obtained from $\beta=(c_+-c_-)/(c_++c_-)$ and $-\Gamma(-\alpha)\cos(\pi\alpha/2)=1/(c_++c_-)$, see~\cite[p.\ 85]{sato}.  
The formula for $\e\X_1^+$ can be found in~\cite[Thm.\ 3]{zolotarev} (stated for an alternative parameterization):
\begin{align}\label{eq:zolotarev}
\e \X^+_1=\frac{\sin (\pi\rho)\Gamma\left(1-1/\a\right)}{\pi |\cos(\pi\alpha(\rho-1/2))|^{1/\alpha}},\qquad\rho=\frac{1}{2}+\frac{\arctan\left(\beta\tan(\pi(\a-2)/2)\right)}{\pi\a}.
\end{align}
For $\beta=0$ we get $\e \X^+_1=\Gamma\left(1-1/\a\right)/\pi$ and for $\beta=\pm 1$ (one-sided jumps) we have
\[\e \X^+_1=\frac{\sin (\pi/\alpha)\Gamma\left(1-1/\a\right)}{\pi |\cos(\pi\alpha/2)|^{1/\alpha}}.\]
It is noted that~\cite[Thm.\ 4.2.1]{chen_thesis} provides a complete asymptotic expansion of $\e V_n$, but only for the symmetric case. 
Indeed, we recover the first term in this expansion using Proposition~\ref{prop:spitzer} and the above formula for $\beta=0$. 

In the setting of Proposition~\ref{prop:spitzer} it is to be expected that $\e V_n\to \e V$, which would then yield an explicit formula for $\e V$.  The missing 
technical detail is uniform integrability of $V_n$'s, which is non-trivial to establish even for a self-similar process $\X$. In this respect, we note  that~\cite[Thm.\ 4.2]{giles_xia} provides some asymptotic bounds on the moments of the error $\e(n^{-1/\a}V_n)^p$ for $p\geq 1$, but those are not strong enough to show that 
$\e V_n^p$ are bounded for some $p>1$.

\bigskip
Importantly,~\eqref{eq:V} provides a way to simulate an approximate realization of $V$ by simulating 
\begin{equation}\label{0107b}\max_{i=0,\ldots, mn} \X_{i/m}-\max_{i=0,\ldots, n} \X_{i}\end{equation}
for large integers $m,n$.
To illustrate the dependence of $V$ on $\alpha\in(0,1)\cup(1,2)$ and $\beta\in[-1,1]$, we took $m=n=100$
and  simulated 50.000\footnote{The exception was the case $\alpha=0.85,\,\beta=0.5$ where
for some reason 1.000.000 replications were needed to get even the present degree
of smoothness.} replications of~\eqref{0107b} (using the standard
Chambers-Mallows-Stuck algorithm for generating stable r.v.'s).  
Based on this, an estimate of the density of $V$ was evaluated via Matlab's
standard kernel smoothing procedure {\tt ksdensity} with default parameters.

The results are given in Figures~\ref{VFigDa}--\ref{VFigDb} for various combinations of $\alpha$ and $|\beta|$, since $\pm\beta$ lead to the same distribution of $V$ according to Proposition~\ref{Prop:0107a}.
The overall picture from these figures and further experiments
is that 
for smaller $\a$ the distribution of $V$ becomes extremely variable
with much heavier tails  (as expected) and also more mass accumulating
close to the origin. As $\alpha\uparrow 2$, the role of $\beta$ becomes less prominent.
Also, one approaches (again as expected) the 
Brownian $V$ and already for $\alpha>1.5$, the difference is not that substantial.
In conclusion, smaller $\a$ leads to a better rate (since $1/a_{1/n}$ is regularly varying at $\infty$ with index $1/\a$) but more disperse limit~$V$.

 \begin{figure}[htb] 
\begin{minipage}{.50\textwidth}
  \centering
  \includegraphics[width=\linewidth]{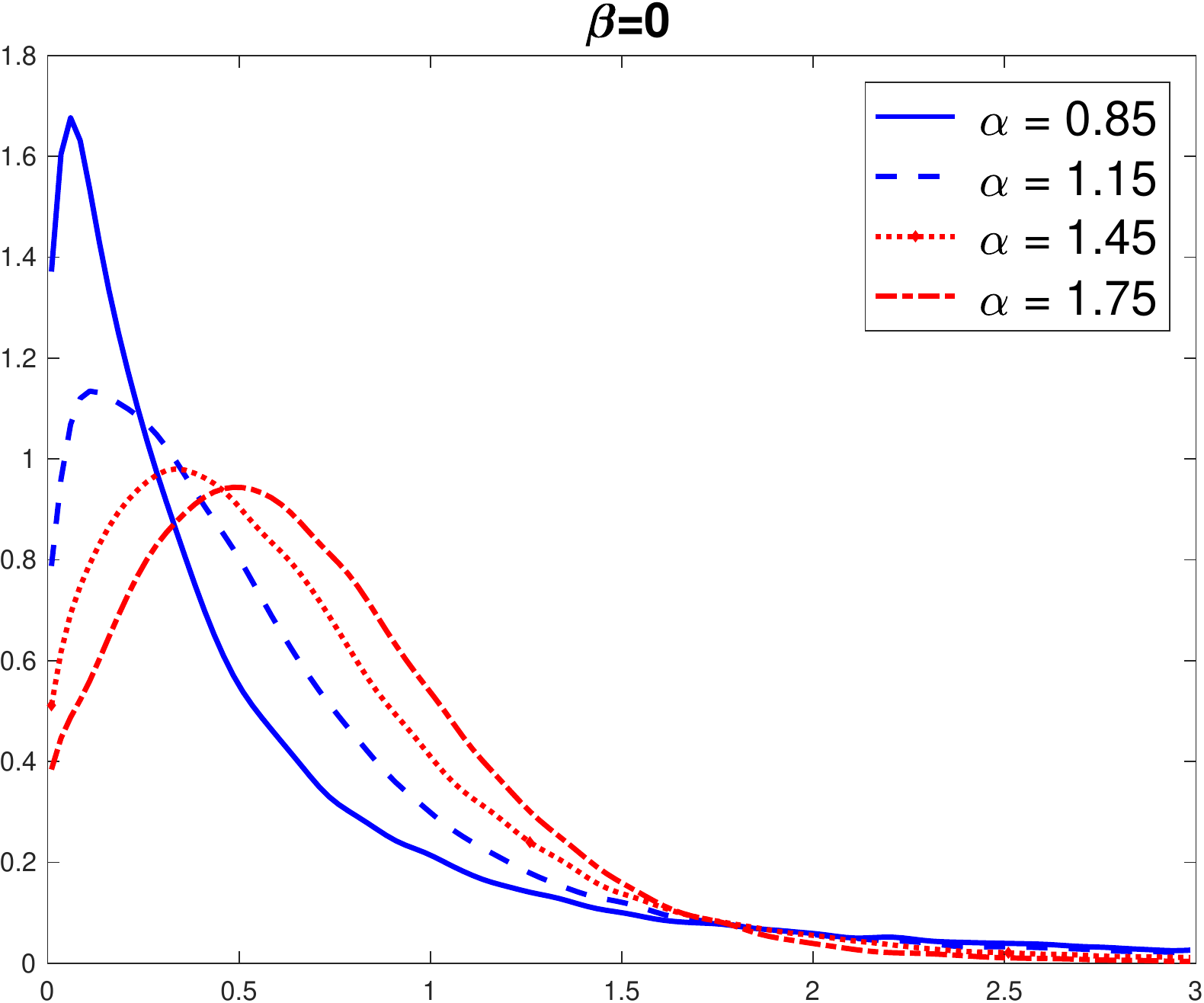}
  \label{VFigDa1}
\end{minipage}%
\begin{minipage}{.50\textwidth}
  \centering
  \includegraphics[width=\linewidth]{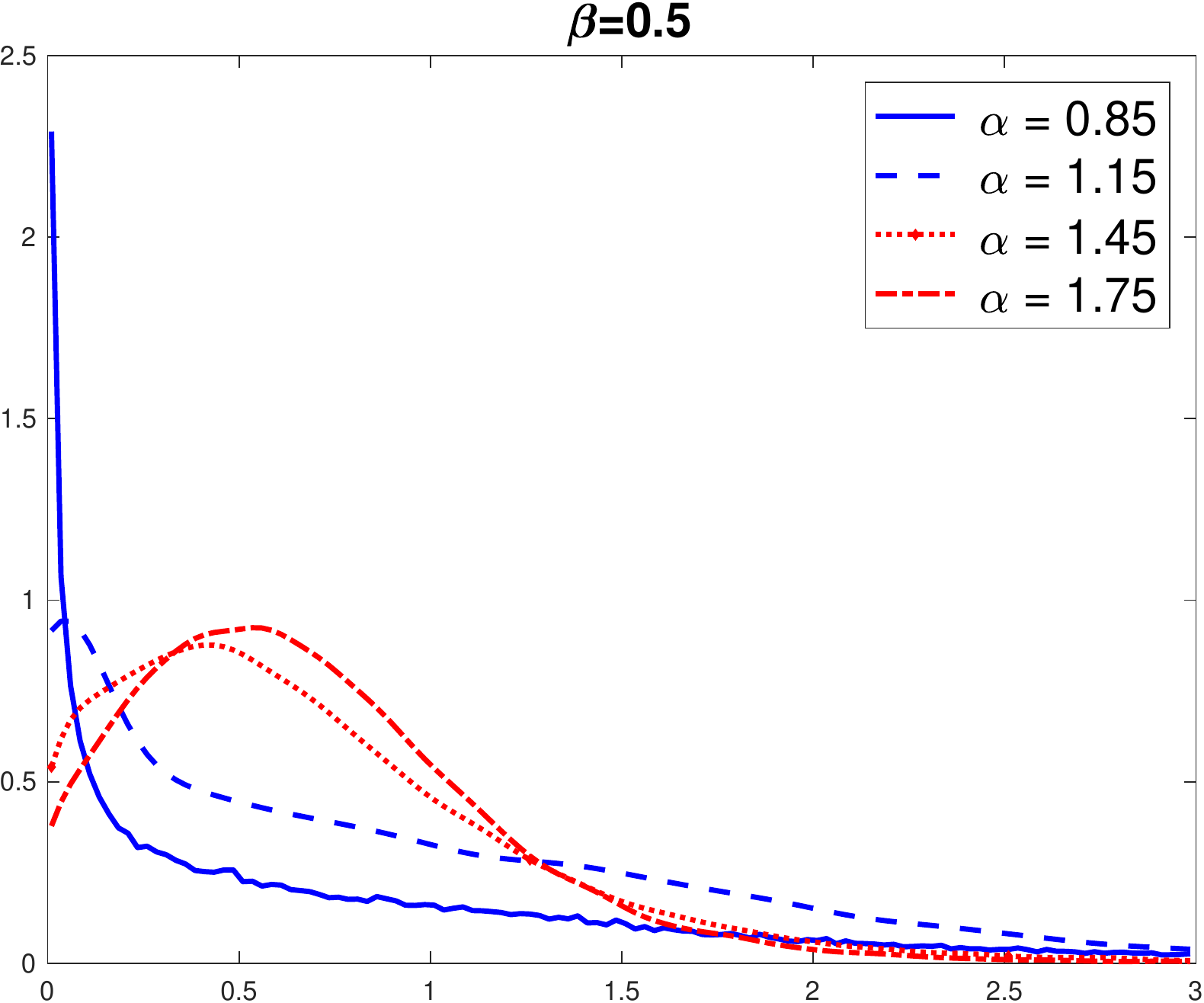}
  \label{VFigDa2}
\end{minipage}%
\caption{Estimated densities of $V$: role of $\alpha$}
\label{VFigDa}
\end{figure}

 \begin{figure}[htb] 
\begin{minipage}{.50\textwidth}
  \centering
  \includegraphics[width=\linewidth]{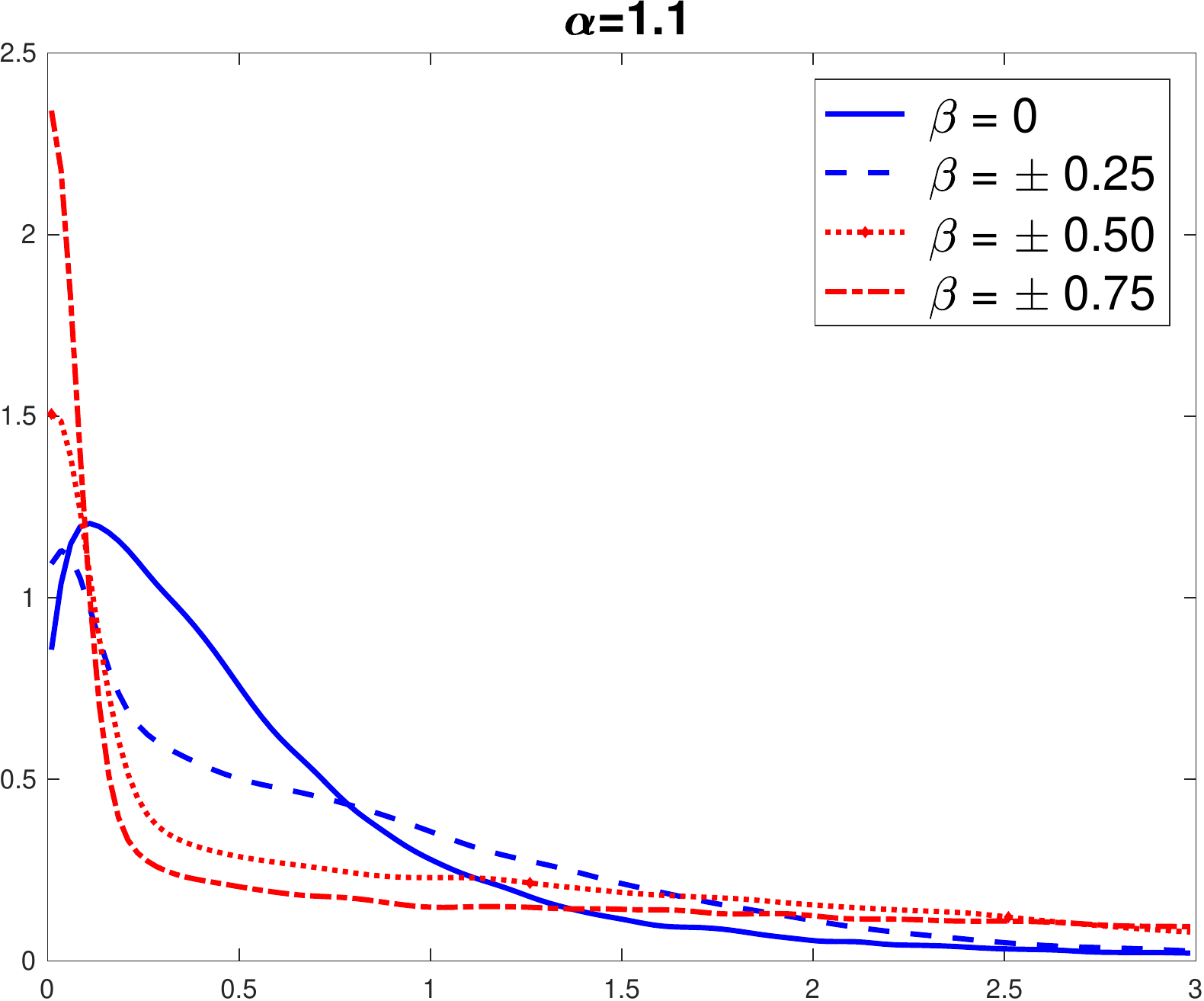}
  \label{VFigDb1}
\end{minipage}%
\begin{minipage}{.50\textwidth}
  \centering
  \includegraphics[width=\linewidth]{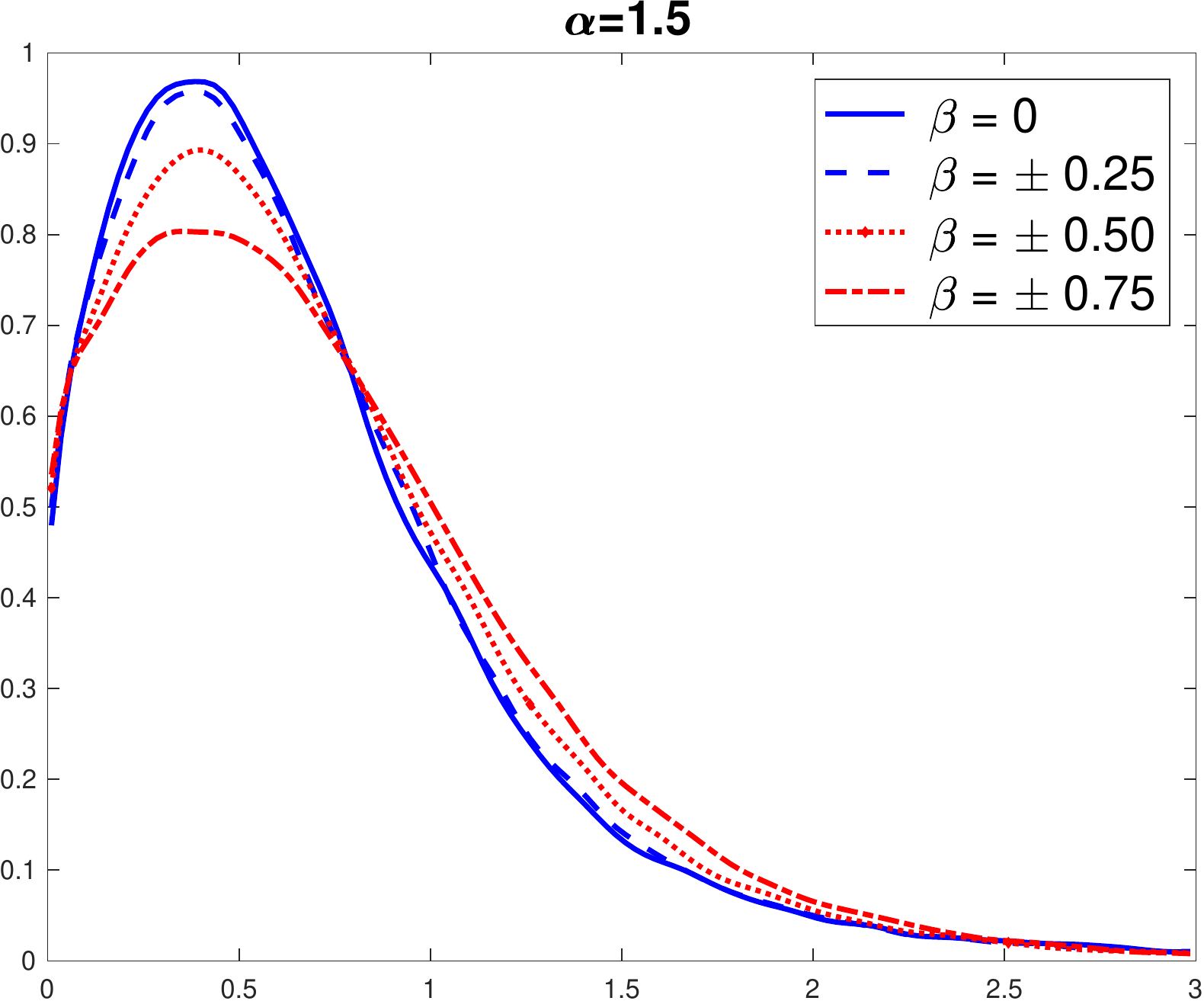}
  \label{VFigDb2}
\end{minipage}%
\caption{Estimated densities of $V$: role of $\beta$}
\label{VFigDb}
\end{figure}

\subsection*{Acknowledgements}
This paper is dedicated to Ward Whitt in appreciation
of his fundamental contributions to applied probability
over several decades. The study links to Ward's work
in several directions. One  is his work \cite{AW} (with Abate)
on one-sided reflection, another is weak convergence
as lucidly exposed in his book \cite{whitt}. Also,
the martingale technique developed by him
and Kella in \cite{KellaW} was the key tool in our first study \cite{Mats} of
two-sided reflection for L\'evy processes.

Additionally, we are thankful to Krzysztof Bisewski for pointing out that the former version of~\eqref{eq:zolotarev} was incorrect.
%
%
%

\end{document}